\documentclass{amsart}
\usepackage{a4wide}
\usepackage[usenames]{color}
\usepackage{multicol}

\newtheorem{lemma}{Lemma}
\newtheorem{theorem}[lemma]{Theorem}
\newtheorem{definition}[lemma]{Definition}
\newtheorem{proposition}[lemma]{Proposition}
\newtheorem{corollary}[lemma]{Corollary}

\newtheorem{remark}[lemma]{Remark}

\newtheorem{conjecture}[lemma]{Conjecture}

\title{Fixed points of morphisms among binary generalized pseudostandard words}

\author{L\!'ubom\'ira Dvo\v r\'akov\'a}
\author{Tereza Velk\'a}
\address{Department of Mathematics, FNSPE, Czech Technical University in Prague,
Trojanova 13, 120~00 Praha~2, Czech Republic}
\email{lubomira.dvorakova@fjfi.cvut.cz, velkater@fjfi.cvut.cz}

\begin{document}

\begin{abstract}
We introduce a class of fixed points of primitive morphisms among aperiodic binary generalized pseudostandard words. We conjecture that this class contains all fixed points of primitive morphisms among aperiodic binary generalized pseudostandard words that are not standard Sturmian words.
\end{abstract}

\maketitle

\section{Introduction}

This note is devoted to binary infinite words generated by a construction called generalized pseudopalindromic closure, known as generalized pseudostandard words.  Concerning this topic, the following facts are known so far: In~\cite{LuDeLu} the generalized pseudostandard words were defined and it was proved there that the famous Thue--Morse word is an example of such words. The authors of~\cite{JaPeSt} characterized generalized pseudostandard words in the class of generalized Thue--Morse words. A~necessary and sufficient condition on periodicity of binary and ternary generalized pseudostandard words was provided in~\cite{BaFl}. The authors of~\cite{MaPa} focused on binary generalized pseudostandard words and obtained several interesting results: an algorithm for a so-called normalization, an effective algorithm for generation of such words, description of generalized pseudostandard words among Rote words.

In this paper, we introduce a new class of aperiodic binary generalized pseudostandard words being fixed points of primitive morphisms.
We moreover conjecture that this is the only class among aperiodic binary generalized pseudostandard words except standard Sturmian words consisting of fixed points of primitive morphisms.

\section{Preliminaries}
Throughout the paper, $\mathbb N$ denotes natural numbers, i.e., $\mathbb N=\{1,2,3,\ldots\}$, while ${\mathbb N}_0=\{0\} \cup \mathbb N$. We restrict ourselves to the binary \emph{alphabet} $\{0,1\}$, we call $0$ and $1$ \emph{letters}. A~\emph{(finite) word} $w$ over $\{0,1\}$ is any finite binary sequence. Its length $|w|$ is the number of letters $w$ contains.
The empty word -- the neutral element for concatenation of words -- is denoted by $\varepsilon$ and its length is set $|\varepsilon|=0$.
The set of all finite binary words is denoted by ${\{0,1\}}^*$. If a~finite word $w=pvs$, where $p, v, s \in {\{0,1\}}^*$, then $p$ is called a~\emph{prefix} of $w$ and $s$ is called a~\emph{suffix} of $w$.
An \emph{infinite word} $\mathbf u$ over $\{0,1\}$ is any binary infinite sequence.
The set of all infinite words is denoted $\{0,1\}^{\mathbb N}$.
A finite word $w$ is a~\emph{factor} of the infinite word $\mathbf u=u_0u_1u_2\ldots$ with $u_i \in \{0,1\}$ if there exists an index $i\geq 0$ such that $w=u_iu_{i+1}\ldots u_{i+|w|-1}$. Such an index is called an \emph{occurrence} of $w$ in $\mathbf u$.

An infinite word $\mathbf u$ is called \emph{recurrent} if each of its factors occurs infinitely many times in $\mathbf u$.
It is said to be \emph{uniformly recurrent} if for every $n \in \mathbb N$ there exists a~length $r(n)$ such that every factor of length $r(n)$ of $\mathbf u$ contains all factors of length $n$ of $\mathbf u$.
We say that an infinite word ${\mathbf u}$ is \emph{eventually periodic} if there exists $v, w \in \{0,1\}^{*}$ such that ${\mathbf u}=wv^{\omega}$, where $\omega$ denotes an infinite repetition. If $w=\varepsilon$, we call $\mathbf u$ \emph{(purely) periodic}. If $\mathbf u$ is not eventually periodic, $\mathbf u$ is said to be \emph{aperiodic}.
It is not difficult to see that if an infinite word is recurrent and eventually periodic, then it is necessarily purely periodic.

A \emph{morphism} is a~map $\varphi: \{0,1\}^* \rightarrow \{0,1\}^*$ such that for every $v, w \in \{0,1\}^*$ we have $\varphi(vw) = \varphi(v) \varphi(w)$. It is clear that in order to define a~morphism, it suffices to provide letter images.
Application of a~morphism $\varphi$ may be naturally extended to an infinite word $\mathbf u=u_0 u_1 u_2 \ldots \in \{0,1\}^{\mathbb N}$ as $\varphi(\mathbf u)=\varphi(u_0)\varphi(u_1)\varphi(u_2)\ldots$ If $\mathbf u=\varphi(\mathbf u)$ for some infinite or finite word $\mathbf u$, we call $\mathbf u$ a~\emph{fixed point} of $\varphi$. Let us make a trivial observation: any periodic infinite word $\mathbf u=v^{\omega}$ is a fixed point of a morphism (for instance, it suffices to set every letter image equal to $v$).

A~morphism $\varphi$ is said to be \emph{primitive} if there exists $n \in \mathbb N$ such that both $\varphi^n(0)$ and $\varphi^n(1)$ contain both letters $0$ and $1$. It is obvious that if an infinite word $\mathbf u \in \{0,1\}^{\mathbb N}$ satisfies $\varphi(\mathbf u)=\mathbf u$ for some primitive morphism $\varphi$, then there exists a~letter $a \in \{0,1\}$ such that $\varphi(a)$ starts in $a$ and $|\varphi(a)|\geq 2$. We say that $\varphi$ is
\emph{prolongable} on $a$. The fixed point $\mathbf u$ starting in $a$ has evidently $\varphi^n(a)$ as its prefix for all $n \in \mathbb{N}$. We sometimes write $\mathbf u=\lim_{n \to \infty}\varphi^n(a)$. It is known that fixed points of primitive morphisms are uniformly recurrent.

%

An \emph{involutory antimorphism} is a~map $\vartheta: \{0,1\}^* \rightarrow \{0,1\}^*$ such that for every $v, w \in \{0,1\}^*$ we have $\vartheta(vw) = \vartheta(w) \vartheta(v)$ and moreover $\vartheta^2$ equals identity. There are only two involutory antimorphisms over the alphabet $\{0,1\}$: the \emph{reversal (mirror) map} $R$ satisfying $R(0)=0, R(1)=1$, and the \emph{exchange antimorphism} $E$ given by $E(0)=1, E(1)=0$. We use the notation $\overline{0} = 1$ and $\overline{1} = 0$, $\overline{E}=R$ and $\overline{R}=E$. A finite word $w$ is a~\emph{palindrome} (an \emph{$R$-palindrome}) if $w = R(w)$, and $w$ is an $E$-\emph{palindrome} ({\em pseudopalindrome}) if $w = E(w)$.
The palindromic closure $w^R$ of a~word $w$ is the shortest palindrome having $w$ as prefix. Similarly, the pseudopalindromic closure $w^E$ of $w$ is the shortest $E$-palindrome having $w$ as prefix.

\subsection{Generalized pseudostandard words}
Let us underline that we again restrict ourselves only to the binary alphabet $\{0,1\}$.

\begin{definition}
Let $\Delta = \delta_1 \delta_2 \ldots \in \{0,1\}^{\mathbb N}$ and $\Theta = \vartheta_1 \vartheta_2 \ldots \in \{E, R\}^{\mathbb N}$. The infinite word $\mathbf{u}(\Delta, \Theta)$ \emph{generated by the generalized pseudopalindromic closure (or generalized pseudostandard word)} is the word whose prefixes $w_n$ are obtained from the recurrence relation
$$w_{n+1} = (w_n \delta_{n+1})^{\vartheta_{n+1}},$$ $$w_0 = \varepsilon.$$ The sequence $\Lambda = (\Delta, \Theta)$ is called the \emph{directive bi-sequence} of the word $\mathbf{u}(\Delta, \Theta)$.

If $\Theta=R^{\omega}$, the word $\mathbf{u}(\Delta, \Theta)$ is called \emph{$R$-standard}. If it is moreover aperiodic, it is called \emph{standard Sturmian}. If $\Theta=E^{\omega}$, the word $\mathbf{u}(\Delta, \Theta)$ is called \emph{$E$-standard} or \emph{pseudostandard}.
\end{definition}

It is readily seen that generalized pseudostandard words are uniformly recurrent.

The sequence $(w_k)_{k\geq 0}$ of prefixes of a~generalized pseudostandard word ${\mathbf u}(\Delta, \Theta)$ does not have to contain all $E$-palindromic and $R$-palindromic prefixes of ${\mathbf u}(\Delta, \Theta)$. Blondin Mass\'e et al.~\cite{MaPa} introduced the notion of normalization of the directive bi-sequence.

\begin{definition}
A~directive bi-sequence $\Lambda=(\Delta, \Theta)$ of a~generalized pseudostandard word $\mathbf{u}(\Delta, \Theta)$ is called \emph{normalized} if the sequence of prefixes $(w_k)_{k\geq 0}$ of $\mathbf{u}(\Delta, \Theta)$ contains all $E$-palindromic and $R$-palindromic prefixes of ${\mathbf u}(\Delta, \Theta)$.
\end{definition}

The authors of~\cite{MaPa} proved that every directive bi-sequence $\Lambda$ can be normalized, i.e., transformed to such a~form $\widetilde \Lambda$ that the new sequence $(\widetilde{w_k})_{k\geq 0}$ contains already every $E$-palindromic and $R$-palindromic prefix and $\widetilde \Lambda$ generates the same generalized pseudostandard word as~$\Lambda$.

\begin{theorem}\label{thm_norm}
Let $\Lambda = (\Delta, \Theta)$ be a~directive bi-sequence. Then there exists a~normalized directive bi-sequence $\widetilde{\Lambda} = (\widetilde{\Delta}, \widetilde{\Theta})$ such that ${\mathbf u}(\Delta, \Theta) = {\mathbf u}(\widetilde{\Delta}, \widetilde{\Theta})$.

Moreover, in order to normalize the sequence $\Lambda$, it suffices firstly to execute the following changes of its prefix (if it is of the corresponding form):
\begin{itemize}
\item $(a\bar{a}, RR) \rightarrow (a\bar{a}a, RER)$,
\item $(a^i, R^{i-1}E) \rightarrow (a^i\bar{a}, R^iE)$ for $i \geq 1$,
\item $(a^i\bar{a}\bar{a}, R^iEE) \rightarrow (a^i\bar{a}\bar{a}a, R^iERE)$ for $i \geq 1$,
\end{itemize}
and secondly to replace step by step from left to right every factor of the form:
\begin{itemize}
\item $(ab\bar{b}, \vartheta\overline{\vartheta}\overline{\vartheta}) \rightarrow (ab\bar{b}b, \vartheta\overline{\vartheta}\vartheta\overline{\vartheta})$,
\end{itemize}
where $a, b \in \{0,1\}$ and $\vartheta \in \{E,R\}$.
\end{theorem}

A necessary and sufficient condition for periodicity of binary generalized pseudostandard words was found in~\cite{BaFl}.
	
	\begin{theorem}
		\label{VZPSperiodic}
		A binary generalized pseudostandard word $\mathbf u(\Delta, \Theta)$, where $\Delta = \delta_1 \delta_2 \ldots \in \{0,1\}^{\mathbb N}$ and $\Theta = \vartheta_1 \vartheta_2 \ldots \in \{E, R\}^{\mathbb N}$, is periodic if and only if the directive bi-sequence $(\Delta, \Theta)$ satisfies the following condition:
		$$ (\exists a \in \{0,1\})(\exists \vartheta \in \{E,R\})(\exists n_0 \in \mathbb N)(\forall n > n_0, n \in \mathbb N)(\delta_{n+1} = a \Leftrightarrow \vartheta_n = \vartheta).$$
	\end{theorem}

\section{Fixed points of morphisms among generalized pseudostandard words}
The aim of this section is to introduce a new class of aperiodic binary generalized pseudostandard words being fixed points of morphisms.
The only known aperiodic binary examples are so far:
	\begin{enumerate}
\item The Thue--Morse word which is the fixed point of the morphism $0 \mapsto 01, 1 \mapsto 10$ (or $0 \mapsto 0110, 1 \mapsto 1001$) and also the pseudostandard word with the directive bi-sequence $(\Delta, \Theta) = (01^\omega , R(ER)^\omega)$, as described in \cite{LuDeLu}.
\item $R$-standard words were studied in~\cite{DrJuPi}. An $R$-standard word $\mathbf u(\Delta, R^{\omega})$ is periodic if and only if the sequence $\Delta$ is eventually constant, i.e., $\Delta=va^{\omega}$, where $v \in \{0,1\}^*$ and $a \in \{0,1\}$. An aperiodic $R$-standard (standard Sturmian) word $\mathbf u(\Delta, R^{\omega})$ is a~fixed point of a~morphism if and only if the sequence $\Delta$ is periodic.   		
\end{enumerate}
	
\subsection{A new class of fixed points of morphisms among aperiodic binary generalized pseudostandard words}
	
	We will now study morphisms $\varphi_{k}$, for $k \in \mathbb{N}$, of the following form:
	\begin{equation}
	\label{eq:phik}
	\varphi_{k} :
	\begin{array}{l}
	0 \mapsto 0(110)^k,   \\
	1 \mapsto 1(001)^k.
	\end{array}
	\end{equation}
	Such a morphism $\varphi_k$ has evidently two fixed points, $\lim_{n \to \infty} \varphi_k^n(0)$ and $\lim_{n \to \infty} \varphi_k^n(1)$. We will prove that the first fixed point is a generalized pseudostandard word whose directive bi-sequence equals $(\Delta, \Theta)=(01^{\omega}, R(ER^{k})^{\omega})$ and the second fixed point has the directive bi-sequence $(\Delta, \Theta)=(10^{\omega}, R(ER^{k})^{\omega})$.
	
	First of all, given the fact that $\varphi_{k}(0)= \overline{\varphi_{k}(1)}$, the fixed point $\lim_{n \to \infty} \varphi_k^n(0)$ equals $\overline{\lim_{n \to \infty} \varphi_k^n(1)}$ and therefore, without loss of generality, we will study the fixed point $\mathbf{u}= \lim_{n \to \infty} \varphi_k^n(0)$ and then easily generalize the results for both cases.
	
	Let us consider the directive bi-sequence $(\Delta, \Theta)=(01^{\omega}, R(ER^{k})^{\omega})$. Note that this directive bi-sequence is normalized by Theorem~\ref{thm_norm}. Remark also that for every $k \in \mathbb N$ the corresponding word $\mathbf{u}=\mathbf{u}(\Delta, \Theta)$ is aperiodic by Theorem~\ref{VZPSperiodic}. We will show in two steps a relation between the prefixes $w_k$ obtained by the generalized pseudopalindromic closure and the morphism $\varphi_{k}$.
	
	Firstly, we will show a recursive relation concerning the prefixes $w_k$ of the word $\mathbf{u}$.  Let $K$ denote the length of the shortest period of $\Theta$, i.e., $K=|RER^{k-1}|= k+1$. We will show that the construction of these prefixes can be expressed depending on $K$.
	
	\begin{lemma}
		\label{Lwl}
		Let $(\Delta, \Theta) = (01^{\omega}, R(ER^{k})^{\omega})$ be the directive bi-sequence of the word $\mathbf{u}$. Let $w_i$, $i \in \mathbb{N}$, be its prefixes obtained by the generalized pseudopalindromic closure and $K= k+1$ the length of the shortest period of $\Theta$. Then for all $l \in {\mathbb N}_0$ the following holds:
		\begin{align*}
			w_{l\cdot K +2} & = w_{l\cdot K +1} \ E(w_{l\cdot K + 1}) = w_{lK+1}\overline{w_{lK+1}} & (E\text{-palindrome})\\
			w_{l\cdot K +3} & = w_{l\cdot K +2} \ R(w_{l\cdot K + 2}) = w_{lK+2}\overline{w_{lK+2}} & (R\text{-palindrome})\\
			w_{l\cdot K +4} & = w_{l\cdot K +3} \ w^{-1}_{l\cdot K +1} \ w_{l\cdot K+3} & (R\text{-palindrome})\\
			w_{l\cdot K +5} & = w_{l\cdot K +4} \ w^{-1}_{l\cdot K +3} \ w_{l\cdot K+4} & (R\text{-palindrome})\\
			\vdots			& \hspace{2.5em} \vdots&	\vdots \\
			w_{l\cdot K + K} & =  w_{l\cdot K + (K-1)} \ w^{-1}_{l\cdot K +(K-2)} \ w_{l\cdot K + (K-1)} & (R\text{-palindrome})\\
			w_{(l+1)\cdot K +1} & = w_{(l+1)\cdot K} \ w^{-1}_{l\cdot K + (K-1)} \ w_{(l+1)\cdot K} & (R\text{-palindrome})\\
		\end{align*}
	\end{lemma}
	\begin{remark}
		The reason why we start from $l\cdot K +2$ is that $w_1$ is the only prefix for which the lemma does not hold. Naturally, if $K = k+1 =$ 2, 3 or 4, only the first 2, 3 or 4 lines hold and then another cycle is started with $l = l+1$.
	\end{remark}
	\begin{proof}
		Let $\Delta = \delta_1 \delta_2 \ldots$ and $\Theta = \vartheta_1 \vartheta_2 \ldots $ We will prove the lemma directly using the construction of pseudostandard words. In each case, we will construct $w_{i+1} = (w_i\delta_{i+1})^{\vartheta_{i+1}}$ after finding the longest $\vartheta_{i+1}$-palindromic suffix of $w_i\delta_{i+1}$. We will split the proof into several cases:
		\begin{enumerate}
			\item Case of $w_{l\cdot K +2}$
			
			If we show that the longest $E$-palindromic suffix of $w_{l\cdot K +1}1$ is $01$, then it is clear that $w_{l\cdot K +2} = w_{l\cdot K +1} \ E(w_{l\cdot K + 1}) = w_{lK+1}\overline{w_{lK+1}}$. The last equality holds because using the form of $\Theta=R(ER^k)^{\omega}$ one can see that $w_{l \cdot K+1}$ is an $R$-palindrome, and applying $E$ on an $R$-palindrome is equivalent to the exchange of zeros and ones.
			
			The longest $E$-palindromic suffix is at least $01$ because $w_{l\cdot K +1}$ is an $R$-palindrome and thus ends with a 0. It cannot be longer than $01$ because that would mean there exists an $E$-palindromic prefix of  $w_{l\cdot K +1}$ followed by a 0 (the prefix $w_{l\cdot K+1}$ is an $R$-palindrome and the reverse image of its $E$-palindromic suffix preceded by a 0 is an $E$-palindromic prefix followed by a 0) and this is not possible since every $E$-palindromic prefix is followed by 1 due to the fact that $\Delta$ is normalized and $\Delta = 01^{\omega}$.
			
			\item Case of $w_{l\cdot K +3}$
			
			Similarly to the case above, if we show that the longest $R$-palindromic suffix of $w_{l\cdot K +2}1$ is $11$, it directly follows that $w_{l\cdot K +3} = w_{l\cdot K +2} \ R(w_{l\cdot K + 2})  = w_{lK+2}\overline{w_{lK+2}}$. The last equality holds because again, applying the antimorphism $R$ on an $E$-palindrome is equivalent to the exchange of zeros and ones.
			
			The factor 11 is an $R$-palindromic suffix of $w_{l\cdot K +2}1$ because $w_{l\cdot K +2}$ ends with a $1$ since it is an $E$-palindrome that begins with $0$. Moreover, every $R$-palindromic suffix of $w_{l\cdot K +2}$ is preceded by a 0 because they are $E$-images of $R$-palindromic prefixes of $w_{l\cdot K+2}$ and all non-empty $R$-palindromic prefixes are followed by 1 because of the form of the directive bi-sequence. Hence, we cannot find a longer $R$-palindromic suffix of $w_{l\cdot K +2}$ preceded by 1 and therefore 11 is the longest palindromic suffix of $w_{l\cdot K +2}1$.
			
			\item Case of $w_{l\cdot K +4}$
			
			In this case, we want to prove that $1w_{l\cdot K +1}1$ is the longest $R$-palindromic suffix of $w_{l\cdot K +3}1$. Then, we have $w_{l\cdot K +4} = w_{l\cdot K +3} \ w^{-1}_{l\cdot K +1} \ w_{l\cdot K+3}$.
			
			Now, let us find the longest $R$-palindromic suffix of $w_{l\cdot K +3}$ preceded by a 1. Because of the normalization of $\Delta$, we know the only pseudopalindromic suffixes are images of the prefixes $w_i$.  It cannot be $w_{l\cdot K +3}$ since it is too long. It cannot be also $R(w_{l\cdot K +2})$ since it is an $E$-palindrome. The longest possibility is now $R(w_{l\cdot K +1})= w_{l\cdot K +1}$ and it is the correct one because it is preceded by $1$ since $w_{l\cdot K +3}$ is an $R$-palindrome and the prefix $w_{l\cdot K+1}$ is followed by 1 according to the form of the directive bi-sequence.
			
			\item Cases of $w_{l\cdot K +5}$ to $w_{(l+1)\cdot K +1}$
			
			In these last cases, we will proceed analogously to the previous case. We want to find the longest $R$-palindromic suffix of $w_{l\cdot K +(i-1)}1$, $i \in \{5,\ldots, K+1\}$. It can be easily seen that it is $1w_{l\cdot K +(i-2)}1$ because of the normalization of $\Delta$. It cannot be $1w_{l\cdot K +(i-1)}1$ (it is too long). The next longest $R$-palindromic suffix of $w_{l\cdot K +(i-1)}$ preceded by a 1 is $w_{l\cdot K +(i-2)}$ and therefore $1w_{l\cdot K +(i-2)}1$ is the longest $R$-palindromic suffix of $w_{l\cdot K +(i-1)}1$. As in the previous case, it follows that $w_{l\cdot K +i} = w_{l\cdot K + (i-1)} \ w^{-1}_{l\cdot K +(i-2)} \ w_{l\cdot K+(i-1)}$.
			
		\end{enumerate}
	\end{proof}
	
	\begin{table}[h!]
		\centering
		\begin{tabular}{l|l|l|l|}
			\cline{2-4}
			& \begin{tabular}[c]{@{}l@{}}$k = 1$: \rule{0pt}{2.6ex}\\ $\Theta = R(ER^1)^{\omega}$\end{tabular}                                    & \begin{tabular}[c]{@{}l@{}}$k = 2$: \rule{0pt}{2.6ex}\\ $\Theta = R(ER^2)^{\omega}$\end{tabular}                  & \begin{tabular}[c]{@{}l@{}}$k = 3$:\rule{0pt}{2.6ex}\\ $\Theta = R(ER^3)^{\omega}$\end{tabular}       \\ \hline
			\multicolumn{1}{|l|}{$w_1$} & 0\rule{0pt}{2.6ex}                                                                                                                  & 0\rule{0pt}{2.6ex}                                                                                                & 0\rule{0pt}{2.6ex}                                                                                    \\ \hline
			\multicolumn{1}{|l|}{$w_2$} & 01\rule{0pt}{2.6ex}                                                                                                                 & 01\rule{0pt}{2.6ex}                                                                                               & 01\rule{0pt}{2.6ex}                                                                                   \\ \hline
			\multicolumn{1}{|l|}{$w_3$} & 0110\rule{0pt}{2.6ex}                                                                                                               & 0110\rule{0pt}{2.6ex}                                                                                             & 0110\rule{0pt}{2.6ex}                                                                                 \\ \hline
			\multicolumn{1}{|l|}{$w_4$} & 01101001\rule{0pt}{2.6ex}                                                                                                           & 0110110\rule{0pt}{2.6ex}                                                                                          & 0110110\rule{0pt}{2.6ex}                                                                              \\ \hline
			\multicolumn{1}{|l|}{$w_5$} & 0110100110010110\rule{0pt}{2.6ex}                                                                                                   & 01101101001001\rule{0pt}{2.6ex}                                                                                   & 0110110110\rule{0pt}{2.6ex}                                                                           \\ \hline
			\multicolumn{1}{|l|}{$w_6$} & \begin{tabular}[c]{@{}l@{}}\rule{0pt}{2.6ex} 0110100110010110\\ 1001011001101001\end{tabular}                                       & \begin{tabular}[c]{@{}l@{}}\rule{0pt}{2.6ex} 01101101001001\\ 10010010110110\end{tabular}                         & 01101101101001001001\rule{0pt}{2.6ex}                                                                 \\ \hline
			\multicolumn{1}{|l|}{$w_7$} & \begin{tabular}[c]{@{}l@{}}\rule{0pt}{2.6ex} 0110100110010110\\ 1001011001101001\\ 1001011001101001\\ 0110100110010110\end{tabular} & \begin{tabular}[c]{@{}l@{}}\rule{0pt}{2.6ex} 01101101001001\\ 10010010110110\\ 100100110010010110110\end{tabular} & \begin{tabular}[c]{@{}l@{}}\rule{0pt}{2.6ex} 01101101101001001001\\ 10010010010110110110\end{tabular} \\ \hline
		\end{tabular}
		\caption{The first seven prefixes $w_i$ of the word $\mathbf u=\mathbf u(01^{\omega}, R(ER^{k})^{\omega})$ for $k=1$ (the Thue--Morse word), $2$ and $3$. }
		\label{Texk}
	\end{table}

	Secondly, we will prove a proposition showing a relation between the prefixes of $\mathbf{u} = \mathbf{u}(\Delta, \Theta)$ and the prefixes of the fixed point of $\varphi_{k}$. The fact that $\mathbf{u}$ is a fixed point of $\varphi_{k}$ follows then immediately from this proposition.
	\begin{proposition}\label{w_and_varphi}
		Let $(\Delta, \Theta) = (01^{\omega}, R(ER^{k})^{\omega})$ be the directive bi-sequence of the word $\mathbf{u}$. Let $w_i$, $i \in \mathbb{N}$, be its prefixes obtained by the generalized pseudopalindromic closure and $K= k+1$ the length of the shortest period of $\Theta$. Let $\varphi_k$ be the morphism defined in \eqref{eq:phik}. Then for every $l \in {\mathbb N}$:
		\begin{equation}
		\label{eq:wkphi}
			w_{l\cdot K + r} = \varphi_{k} (w_{(l-1)\cdot K + r})
		\end{equation}
		for all $r \in \{1,2,\ldots, K \}$.
	\end{proposition}
	\begin{proof}
		We proceed by induction on $l$.
		
		First, we will prove that if the equality \eqref{eq:wkphi} holds for some $l-1$, it also holds for $l$. The proof for $l=1$ will be then easier because during the proof of the induction step, we will show that for a fixed $l$, if the equality  holds for $r=1$, then it easily follows that \eqref{eq:wkphi} holds for all $r \in \{2,\ldots,K\}$.
		
We rewrite the prefixes $w_i$ using Lemma \ref{Lwl} and we apply the fact that $\varphi_k(\overline{w})=\overline{\varphi_k(w)}$.

$\begin{array}{rcl}w_{l\cdot K + 1}&=& w_{l\cdot K } \ w^{-1}_{l\cdot K - 1} \ w_{l\cdot K} \\
&=&  w_{(l-1)\cdot K + K } \ w^{-1}_{(l-1)\cdot K + K - 1} \ w_{(l-1)\cdot K + K} \\
&=&\varphi_k(w_{(l-2)\cdot K + K }) \left(\varphi_k( w_{(l-2)\cdot K + K - 1})\right)^{-1} \varphi_k(w_{(l-2)\cdot K + K}) \\
&=& \varphi_k(w_{(l-2)\cdot K + K } \ w^{-1}_{(l-2)\cdot K + K - 1} \  w_{(l-2)\cdot K + K}) \\
&=& \varphi_k (w_{(l-1)\cdot K +1})\end{array}$
			
$\begin{array}{rcl}w_{l\cdot K + 2}&=& w_{l\cdot K +1} \ \overline{w_{l\cdot K + 1}}\\
& =& \varphi_k (w_{(l-1)\cdot K +1}) \ \overline{\varphi_k (w_{(l-1)\cdot K + 1})} \\
&=&  \varphi_k (w_{(l-1)\cdot K +1} \ \overline{w_{(l-1)\cdot K + 1}}) \\
& =& \varphi_k (w_{(l-1)\cdot K +2})\end{array}$					

$\begin{array}{rcl}w_{l\cdot K + 3} &=& w_{l\cdot K +2} \ \overline{w_{l\cdot K + 2}}\\
 &=& \varphi_k (w_{(l-1)\cdot K +2}) \ \overline{\varphi_k (w_{(l-1)\cdot K +2})}\\
  &=& \varphi_k (w_{(l-1)\cdot K +2} \ \overline{w_{(l-1)\cdot K +2}})\\
   &=& \varphi_k (w_{(l-1)\cdot K + 3})
   \end{array}$
			
$\begin{array}{rcl} w_{l\cdot K + 4}& =& w_{l\cdot K +3} \ w^{-1}_{l\cdot K +1} \ w_{l\cdot K+3} \\
&=& \varphi_k (w_{(l-1)\cdot K + 3}) \ \left(\varphi_k (w_{(l-1)\cdot K +1})\right)^{-1} \ \varphi_k (w_{(l-1)\cdot K + 3}) \\
&=& \varphi_k (w_{(l-1)\cdot K + 3} \ w_{(l-1)\cdot K +1}^{-1} \ w_{(l-1)\cdot K + 3}) \\
&=& \varphi_k (w_{(l-1)\cdot K + 4}) \end{array}$
			
\noindent Consider $w_{l\cdot K + i}$, $i = \{5, \ldots, K\}$: We suppose that we first prove the proposition for $i = 5$, then $6$ etc. We proceed exactly in the same way as in the cases above.
			
$\begin{array}{rcl}w_{l\cdot K +i} &=& w_{l\cdot K + (i-1)} \ w^{-1}_{l\cdot K +(i-2)} \ w_{l\cdot K+(i-1)} \\
&=& \varphi_k (w_{(l-1)\cdot K + (i-1)} \ w^{-1}_{(l-1)\cdot K +(i-2)} \ w_{(l-1)\cdot K+(i-1)})\\
&=& \varphi_k (w_{(l-1)\cdot K + (i-1)})\left(\varphi_k(w_{(l-1)\cdot K +(i-2)})\right)^{-1} \varphi_k( w_{(l-1)\cdot K+(i-1)})\\
&=& \varphi_k (w_{(l-1)\cdot K +i})\end{array}$
			
\noindent Now, let us focus on the case $l=1$. If we prove the proposition for $r=1$, we will have that $w_{K+1} = \varphi_k (w_1)$. Moreover, Lemma $\ref{Lwl}$ holds for all $w_i$, $i \geq 2$, and therefore, all above equalities are also satisfied if we set $l= 1$, so the proposition holds for $r = \{2, \ldots, K\}$.
		
		Thus, the last case we need to prove is the statement for $l=1$ and $r=1$, i.e., $\varphi_k(w_1)=w_{K+1}$. By the definition of $\varphi_k$, we have $\varphi_k (w_1) = 0(110)^k = 0(110)^{K-1}$. We have to show that $0(110)^{K-1}$ is equal to $w_{K+1}$. We have $w_1 = 0$, $w_2 = 01$, $w_3 = (w_21)^R = 0110 = 0(110)^1$, $w_4 = (w_31)^R = 0110110 = 0(110)^2$ and proceeding in the same way, i.e., adding 1 to the end and making the $R$-palindromic closure, we obtain $ w_{K+1}=0(110)^{K-1}$.
		
	\end{proof}
By Proposition~\ref{w_and_varphi}, we get the final corollary providing a new class of binary generalized pseudostandard words being fixed points of morphisms.
	\begin{corollary}
		Denote $\mathbf{u}=lim_{n \to \infty} \varphi_k^n(0)$ and $\mathbf{v}=lim_{n \to \infty} \varphi_k^n(1)$, where $\varphi_k $ is defined in \eqref{eq:phik}. Then $\mathbf{u}=\mathbf{u}(01^\omega, R(ER^k)^\omega)$ and $\mathbf{v}=\mathbf{v}(10^\omega, R(ER^k)^\omega)$.
	\end{corollary}

\section{Open problems}
According to our computer experiments, it seems that the morphisms defined in~\eqref{eq:phik} are the only primitive morphisms whose fixed points are aperiodic binary generalized pseudostandard words that are not standard Sturmian words. Let us state thus the following conjecture.
\begin{conjecture}
Let $\mathbf u$ be an aperiodic binary generalized pseudostandard word, not standard Sturmian, being a~fixed point of a~primitive morphism. Then $\mathbf u=\varphi_k(\mathbf u)$ for some $k \in \mathbb N$, where $\varphi_k$ is the morphism defined in~\eqref{eq:phik}.
\end{conjecture}

In this note, we focused only on binary words. Generalized pseudostandard words are defined over any alphabet $\{0,1, \ldots, m-1\}$ for $m>1, \ m \in \mathbb N$~\cite{LuDeLu}. To our knowledge, the only fact known about fixed points of primitive morphisms in a~multiliteral case is a~result from~\cite{JaPeSt}:
The generalized Thue--Morse words are defined for $m,b \in {\mathbb N}$, $m>1$, $b>1$ as follows:
		$$\mathbf{t}_{b,m} = \big{(} s_b(n) \bmod m\big)^{\infty}_{n=0}$$
		where $s_b(n)$ denotes the digit sum of the expansion of number $n$ in the base $b$. Such words are fixed points of morphisms. A~generalized Thue--Morse word ${\mathbf{t}}_{b,m}$ is a generalized pseudostandard word if and only if $b \leq m$ or $b-1 = 0 \pmod{m}$. Note that the Thue--Morse word is a special case of $\mathbf{t}_{b,m}$ for $b = m = 2$. For the form of morphisms whose fixed points are the generalized Thue--Morse words and for the form od their directive bi-sequence $(\Delta, \Theta)$ and other properties, see \cite{JaPeSt}. Hence, it is an open problem to detect fixed points of primitive morphisms over larger alphabets in the class of aperiodic generalized pseudostandard words.

\section*{Acknowledgements}
We acknowledge the financial support of Czech Science
Foundation GA\v{C}R 13-03538S.


\end{document}